\documentclass[12pt]{article}
\usepackage{amsmath}
\usepackage{amsfonts}
\usepackage{amssymb}
\usepackage{verbatim}
\usepackage{showlabels}
\usepackage[T1]{fontenc}
\usepackage[utf8]{inputenc}
\usepackage{authblk}
\usepackage[
bookmarks=true,         bookmarksnumbered=true,                        colorlinks=true,
pdfstartview=FitV,
linkcolor=blue, citecolor=blue, urlcolor=blue]{hyperref}

\newtheorem{theorem}{Theorem}

\newtheorem{lemma}[theorem]{Lemma}

\newenvironment{proof}[1][Proof]{\noindent\textit{#1.} }{\ \rule{0.5em}{0.5em}}
\begin{document}

\title{Non-degeneracy of some Sobolev Pseudo-norms of fractional Brownian
motion}
\author[1]{Yaozhong Hu}
\author[2]{Fei Lu}
\author[1]{David Nualart}
\affil[1]{Department of Mathematics, University of Kansas }
\affil[2]{Mathematics Department, Lawrence Berkeley National Laboratory}

\renewcommand\Authands{ and }
%\author{Yaozhong Hu, \ Fei Lu \ and \ David Nualart \\
%EndAName
%Department of Mathematics \\
%University of Kansas \\
%Lawrence, Kansas, 66045 USA}
%\date{July 4, 2013}
\maketitle

\begin{abstract}
Applying an upper bound estimate for small $L^{2}$ ball probability for
fractional Brownian motion (fBm), we prove the non-degeneracy of some
Sobolev pseudo-norms of fBm.
\end{abstract}


% MSC 2010 classification: 
%     60H07:  Stochastic calculus of variations and the Malliavin calculus;
%     60G22: Fractional processes, including fractional Brownian motion
\section{Introduction}

Let $B^{H}=\left\{ B_{t}^{H}:t\in \lbrack 0,1]\right\} $ be a fractional
Brownian motion (fBm) on $\left( \Omega ,\mathcal{F},P\right) $. That is, $%
\left\{ B_{t}^{H}:t\geq 0\right\} $ is a centered Gaussian process with
covariance%
\begin{equation}
R_{H}(t,s)=E(B_{t}^{H}B_{s}^{H})=\frac{1}{2}(\left\vert t\right\vert
^{2H}+\left\vert s\right\vert ^{2H}-\left\vert t-s\right\vert ^{2H}),
\label{CovfBm}
\end{equation}%
where $H\in (0,1)$ is the Hurst parameter. Consider the random variable $F$
given by a functional of $B^{H}$:%
\begin{equation}
F=\int_{0}^{1}\int_{0}^{1}\frac{\left\vert B_{t}^{H}-B_{t^{\prime
}}^{H}\right\vert ^{2p}}{\left\vert t-t^{\prime }\right\vert ^{q}}%
dtdt^{\prime },  \label{Sob-norm}
\end{equation}%
where $p,q\geq 0$ satisfy $(2p-2)H>q-1$.

In the case of $H=\frac{1}{2}$, $B^{H}$ is a Brownian motion, and the random
variable $F$ is the Sobolev norm on the Wiener space considered by Airault
and Malliavin in \cite{AM88}. This norm plays a central role in the
construction of surface measures on the Wiener space. Fang \cite{Fang91}
showed that $F$ is non-degenerate in the sense of Malliavin calculus (see
the definition below). Then it follows from the well-known criteria on
regularity of densities that the law of $F$ has a smooth density.

The purpose of this note is to extend this result to the case $H\neq \frac{1%
}{2}$ and to show that $F$ is non-degenerate. 

In order to state our result precisely, we need some notations from
Malliavin calculus (for which we refer to Nualart \cite[Section 1.2]{Nu06}).
Denote by $\mathcal{E}$ the set of all step functions on $[0,1]$. Let $%
\mathfrak{H}$ be the Hilbert space defined as closure of $\mathcal{E}$ with
respect to the scalar product%
\begin{equation*}
\langle \mathbf{1}_{[0,t]},\mathbf{1}_{[0,s]}\rangle _{\mathfrak{H}%
}=R_{H}(t,s),\text{ for }s,t\in \lbrack 0,1].
\end{equation*}%
Then the mapping $\mathbf{1}_{[0,t]}\mapsto B_{t}^{H}$ extends to a linear
isometry between $\mathfrak{H}$ and the Gaussian space spanned by $B^{H}$.
We denote this isometry by $B^{H}$. Then, for any $h,g\in \mathfrak{H}$, $%
B^{H}(f)$ and $B^{H}(g)$ are two centered Gaussian random variables with $%
E[B^{H}(h)B^{H}(g)]=\left\langle h,g\right\rangle _{\mathfrak{H}}$. We
define the space $\mathbb{D}^{1,2}$ as the closure of the set of smooth and
cylindrical random variable of the form 
\begin{equation*}
G=f(B^{H}(h_{1}),\dots ,B^{H}(h_{n}))
\end{equation*}%
with $h_{i}\in \mathfrak{H}$, $f\in C_{p}^{\infty }(\mathbb{R}^{n})$ ($f$
and all its partial derivatives has polynomial growth) under the norm%
\begin{equation*}
\left\Vert G\right\Vert _{1,2}=\sqrt{E[G^{2}]+E[\left\Vert DG\right\Vert _{%
\mathfrak{H}}^{2}]},
\end{equation*}%
where the $DF$ is the Malliavin derivative of $F$ defined as 
\begin{equation*}
DG=\sum_{i=1}^{n}\frac{\partial f}{\partial x_{i}}(B^{H}(h_{1}),\dots
,B^{H}(h_{n}))h_{i}.
\end{equation*}

We say that a random vector $\mathbf{V=}\left( V_{1},\dots ,V_{d}\right) $
whose components are in $\mathbb{D}^{1,2}$ is \emph{non-degenerate} if its
Malliavin matrix $\gamma _{\mathbf{V}}=\left( \left\langle
DV_{i},DV_{j}\right\rangle _{\mathfrak{H}}\right) $ is invertible a.s. and $%
(\det \gamma _{\mathbf{V}})^{-1}\in L^{p}(\Omega ),$ for all $p\geq 1$ (see
for instance \cite[Definition 2.1.1]{Nu06}). Our main result is the
following theorem.

\begin{theorem}
\label{Main-Non-deg}For all $H\in (0,1)$, the functional $F$ of a fBm $B^{H}$
given in $(\ref{Sob-norm})$ is non-degenerate. That is, 
\begin{equation}
\left\Vert DF\right\Vert _{\mathfrak{H}}^{-1}\in L^{k}(\Omega ),\text{ for
all }k\geq 1.  \label{Non-deg-Th}
\end{equation}
\end{theorem}

We shall follow the same scheme introduced in \cite{Fang91} to prove Theorem %
\ref{Main-Non-deg}. That is, it suffices to prove that for any integer $n$,
there exists a constant $C_{n}$ such that 
\begin{equation}
P(\left\Vert DF\right\Vert _{\mathfrak{H}}\leq \varepsilon )\leq
C_{n}\varepsilon ^{n}  \label{DF-SBP}
\end{equation}%
for all $\varepsilon $ small. This kind of inequality is called upper bound
estimate in small deviation theory (also called small ball probability
theory, for which we refer to \cite{LS01} and the reference therein). To
prove (\ref{DF-SBP}), we will need an upper bound estimate of the small
deviation for the path variance of the fBm, which is introduced in the
following section.

We comment that Li and Shao \cite[Theorem 4]{LS99} proved that 
\begin{equation}
P\left( \int_{0}^{1}\int_{0}^{1}\frac{\left\vert
B_{t}^{H}-B_{s}^{H}\right\vert ^{2p}}{\left\vert t-s\right\vert ^{q}}%
dtds\leq \varepsilon \right) \leq \exp \{-\frac{C}{\varepsilon ^{\beta }}\}
\label{L-S99}
\end{equation}%
for $p>0$, $0\leq q<1+2pH$, $q\neq 1$ and $\beta =1/(pH-\max \left\{
0,q-1\right\} $. But (\ref{L-S99}) gives the small ball probability of $F$,
not\ of $\left\Vert DF\right\Vert _{\mathfrak{H}}$.

\section{An estimate on the path variance of fBm}

In this section we show the following useful lemma.

\begin{lemma}[Estimate of the path variance of the fBm]
\label{PathVar}Let $B^{H}=\left\{ B_{t}^{H}:t\in \lbrack 0,1]\right\} $ be a
fBm. For $0\leq a<b\leq 1$, consider the path variance $V_{[a,b]}(B^{H})$
defined by%
\begin{equation*}
V_{[a,b]}(B^{H})=\int_{a}^{b}\left\vert B_{t}^{H}\right\vert ^{2}\frac{dt}{%
b-a}-(\int_{a}^{b}B_{t}^{H}\frac{dt}{b-a})^{2}.
\end{equation*}%
Then for $c_{H}=H\left( (2H+1)\sin \frac{\pi }{2H+1}\right) ^{-\frac{2H+1}{2H%
}}\left( \Gamma (2H+1)\sin (\pi H)\right) ^{\frac{1}{2H}}$, 
\begin{equation}
\lim_{\varepsilon \rightarrow 0}\varepsilon ^{\frac{1}{H}}\log
P(V_{[a,b]}(B^{H})\leq \varepsilon ^{2})=-(b-a)c_{H}.  \label{PathVar0}
\end{equation}
\end{lemma}

Actually, we will only need 
\begin{equation}
\limsup_{\varepsilon \rightarrow 0}\varepsilon ^{\frac{1}{H}}\log
P(V_{[a,b]}(B^{H})\leq \varepsilon ^{2})<\infty .  \label{PathVarU}
\end{equation}%
In the case of $H=\frac{1}{2}$, this estimate of the path variance for
Brownian motion was introduced by Malliavin \cite[Lemma 3.3.2]{Mal78b},
using the following Payley--Wiener expansion of Brownian motion: 
\begin{equation}
B_{t}=tG+\sqrt{2}\sum_{k=1}^{\infty }\frac{1}{2\pi k}(X_{k}\cos 2\pi
kt+Y_{k}\sin 2\pi kt),\text{ a.s. for all }t\in \lbrack 0,1],  \label{PWe-Bm}
\end{equation}%
where $G$, $X_{k},Y_{k}$, $k\in \mathbb{N}$, are i.i.d. standard Gaussian
random variables. Then the estimate (\ref{PathVarU}) follows by observing
that $V_{[0,1]}(B)=\frac{1}{2\pi ^{2}}\sum_{k=1}^{\infty }\frac{1}{2\pi k}%
(X_{k}^{2}+Y_{k}^{2})$, a sum of $\chi ^{2}(1)$ random variables. The above
expansion of Brownian motion can be obtained by integrating an expansion of
white noise on the orthonormal basis $\left\{ 1,\sqrt{2}\cos 2\pi kt,\sqrt{2}%
\sin 2\pi kt\right\} $ of $L^{2}[0,1]$. Payley--Wiener expansion of fBm has
been established recently by Dzhaparidze and van Zanten \cite{DZ05}: 
\begin{equation}
B_{t}^{H}=tX+\sum_{k=1}^{\infty }\frac{1}{\omega _{k}}\left[ X_{k}(\cos
2\omega _{k}t-1)+Y_{k}\sin 2\omega _{k}t\right] ,  \label{PWe-fBm}
\end{equation}%
where $0<\omega _{1}<\omega _{2}<\dots $ are the real zeros of $J_{-H}$ (the
Bessel function of the first kind of order $-H$), and $X$, $X_{k},Y_{k}$, $%
k\in \mathbb{N}$, are independent centered Gaussian random variables with
variance 
\begin{equation*}
EX^{2}=\sigma _{H}^{2},EX_{k}^{2}=EY_{k}^{2}=\sigma _{k}^{2},
\end{equation*}%
with $\sigma _{H}^{2}=\frac{\Gamma (\frac{3}{2}-H)}{2H\Gamma (H+\frac{1}{2}%
)\Gamma (3-2H)}$ and $\sigma _{k}^{2}=\sigma _{H}^{2}(2-2H)\Gamma
^{2}(1-H)\left( \frac{\omega _{k}}{2}\right) ^{2H}J_{-H}(\omega _{k})$.
Because the path variance $V_{[0,1]}(B^{H})$ is difficult to evaluate in the
case $H\neq \frac{1}{2}$, the techniques of \cite[Lemma 3.3.2]{Mal78b} to
prove (\ref{PathVarU}) no longer work.

Fortunately, recent developments in small deviation theory allow us to
derive a simple proof of  (\ref{PathVar0}).

\begin{proof}[Proof of Lemma \protect\ref{PathVar}]
In \cite[Theorem 3.1 and Remark 3.1]{NaNi05} Nazarov and Nikitin proved that
for any square integrable random variable $G$ and any nonnegative function $%
\psi \in L^{1}[0,1]$, 
\begin{equation}
\lim_{\varepsilon \rightarrow 0}\varepsilon ^{\frac{1}{H}}\log
P(\int_{0}^{1}(B_{t}^{H}-G)^{2}\psi (t)dt\leq \varepsilon ^{2})=-c_{H}\left(
\int_{0}^{1}\psi (t)^{\frac{1}{2H+1}}dt\right) ^{\frac{2H+1}{2H}}.
\label{NN-SBP}
\end{equation}%
Notice that by the self-similarity property of fBm, 
\begin{equation*}
V_{[a,b]}(B^{H})=\int_{a}^{b}\left( B_{t}^{H}-\overline{B^{H}}\right) ^{2}%
\frac{dt}{b-a}=b\int_{a/b}^{1}\left( B_{bu}^{H}-\overline{B^{H}}\right) ^{2}%
\frac{du}{b-a}
\end{equation*}%
has the same distribution as $b^{2H+1}\int_{a/b}^{1}\left( B_{u}^{H}-b^{-H}%
\overline{B^{H}}\right) ^{2}\frac{du}{b-a}$. Then, Lemma \ref{PathVar}
follows from (\ref{NN-SBP}) by taking $G=b^{-H}\overline{B^{H}}$ and $\psi
(t)=\mathbf{1}_{[a/b,1]}(t)$.
\end{proof}

We comment that Bronski \cite{Bro03} proved (\ref{NN-SBP}) for the case $G=0$
and $\psi \equiv 1$ by estimating the asymptotics of the Karhunen--Loeve
eigenvalues of fBm. Actually, the assumption $G=0$ is not necessary, because
a random variable $G$ here doesn't contribute to the asymptotics of the
Karhunen--Loeve eigenvalues.

\section{Proof of the main theorem}

In this section we prove (\ref{Non-deg-Th}) by estimating $P(\left\Vert
DF\right\Vert _{\mathfrak{H}}\leq \varepsilon )$ for $\varepsilon $ small.

For simplicity, we denote%
\begin{eqnarray*}
I &=&\left\{ \left( t,t^{\prime }\right) \in \lbrack 0,1]^{2},t^{\prime
}\leq t\right\} , \\
\vec{t} &=&\left( t,t^{\prime }\right) ,~d\vec{t}=dtdt^{\prime }\text{.}
\end{eqnarray*}

\begin{lemma}
\label{Lemma1}Let $Q(\vec{t},\vec{s})=\langle \mathbf{1}_{[t^{\prime },t]},%
\mathbf{1}_{[s^{\prime },s]}\rangle _{\mathfrak{H}}$. Then the operator $Q$
on $L^{2}(I)$ defined by 
\begin{equation*}
Qf(\vec{t})=\int_{I}Q(\vec{t},\vec{s})f(\vec{s})d\vec{s}\text{, }f\in
L^{2}(I)
\end{equation*}%
is symmetric positive and compact.
\end{lemma}

\begin{proof}
Compactness follows from $Q(\vec{t},\vec{s})\in L^{2}(I\times I)$. The
function $Q(\vec{t},\vec{s})$ is symmetric, so is the operator $Q$. Finally, 
$Q$ is positive because for any $f\in L^{2}(I)$, 
\begin{equation*}
\left\langle Qf,f\right\rangle _{L^{2}(I)}=\int_{I}\int_{I}Q(\vec{t},\vec{s}%
)f(\vec{s})d\vec{s}f(\vec{t})d\vec{t}=\left\Vert \int_{I}\mathbf{1}%
_{[t^{\prime },t]}f(\vec{t})d\vec{t},\right\Vert _{\mathfrak{H}}^{2}.
\end{equation*}
\end{proof}

Then, it follows that $Q$ has a sequence of decreasing eigenvalues $\left\{
\lambda _{n}\right\} _{n\in \mathbb{N}}$, i.e. $\lambda _{1}\geq \dots \geq
\lambda _{n}>0,$ and $\lambda _{n}\rightarrow 0$. The corresponding
normalized eigen-functions $\left\{ \varphi _{n}\right\} _{n\in \mathbb{N}}$
form an orthonormal basis of $L^{2}(I)$. Each of them is continuous because $%
\phi _{n}(\vec{t})=\lambda _{n}^{-1}\int_{I}Q(\vec{t},\vec{s})\phi _{n}(\vec{%
s})d\vec{s}$ and $Q(\vec{t},\vec{s})$ is continuous. We can write 
\begin{equation}
Q(\vec{t},\vec{s})=\sum_{n\geq 1}\lambda _{n}\varphi _{n}(\vec{t})\varphi
_{n}\left( \vec{s}\right) .  \label{Qexpan}
\end{equation}

From the definition of Malliavin derivative we have 
\begin{equation*}
D_{r}F=4p\int_{I}\frac{\left\vert B_{t}^{H}-B_{t^{\prime }}^{H}\right\vert
^{2p-1}}{\left\vert t-t^{\prime }\right\vert ^{q}}\mathrm{sign}%
(B_{t}^{H}-B_{t^{\prime }}^{H})\mathbf{1}_{[t^{\prime },t]}(r)dtdt^{\prime }.
\end{equation*}%
Then%
\begin{eqnarray}
\left\Vert DF\right\Vert _{\mathfrak{H}}^{2} &=&16p^{2}\left\Vert \int_{I}%
\mathbf{1}_{[t^{\prime },t]}(\cdot )\frac{\left\vert B_{t}^{H}-B_{t^{\prime
}}^{H}\right\vert ^{2p-1}}{\left\vert t-t^{\prime }\right\vert ^{q}}\mathrm{%
sign}(B_{t}^{H}-B_{t^{\prime }}^{H})dtdt^{\prime }\right\Vert _{\mathfrak{H}%
}^{2}  \label{DF-Hnorm} \\
&=&16p^{2}\int_{I\times I}\langle \mathbf{1}_{[t^{\prime },t]},\mathbf{1}%
_{[s^{\prime },s]}\rangle _{\mathfrak{H}}\frac{\left\vert
B_{t}^{H}-B_{t^{\prime }}^{H}\right\vert ^{2p-1}}{\left\vert t-t^{\prime
}\right\vert ^{q}}\mathrm{sign}(B_{t}^{H}-B_{t^{\prime }}^{H})  \notag \\
&&~~\times \frac{\left\vert B_{s}^{H}-B_{s^{\prime }}^{H}\right\vert ^{2p-1}%
}{\left\vert s-s^{\prime }\right\vert ^{q}}\mathrm{sign}(B_{s}^{H}-B_{s^{%
\prime }}^{H})d\vec{t}d\vec{s}.  \notag
\end{eqnarray}%
Using (\ref{Qexpan}) to evaluate the inner product in (\ref{DF-Hnorm})
yields 
\begin{equation}
\left\Vert DF\right\Vert _{\mathfrak{H}}^{2}=16p^{2}\sum_{i\geq 1}\lambda
_{i}V_{i}^{2},  \label{DF-norm}
\end{equation}%
where we denote 
\begin{equation}
V_{i}=\int_{I}\varphi _{i}(t,t^{\prime })\frac{\left\vert
B_{t}^{H}-B_{t^{\prime }}^{H}\right\vert ^{2p-1}}{\left\vert t-t^{\prime
}\right\vert ^{q}}\mathrm{sign}(B_{t}^{H}-B_{t^{\prime }}^{H})dtdt^{\prime }.
\label{vi}
\end{equation}

For each $\beta =(\beta _{1},\dots ,\beta _{n})\in S^{n-1}$ (the unit sphere
in $\mathbb{R}^{n}$), let $\Psi _{\beta }(\vec{t})=\sum_{i=1}^{n}\beta
_{i}\varphi _{i}(\vec{t})$. We denote 
\begin{equation}
G_{\beta }=\int_{I}\Psi _{\beta }^{2}(\vec{t})\frac{\left\vert
B_{t}^{H}-B_{t^{\prime }}^{H}\right\vert ^{2p-2}}{\left\vert t-t^{\prime
}\right\vert ^{q}}d\vec{t}.  \label{GbetaDef}
\end{equation}

\begin{lemma}
\label{Gbeta}There exists a constant $C_{p,H}>0$ such that for all $\beta
\in S^{n-1}$ and $\varepsilon >0$, 
\begin{equation}
P\left( G_{\beta }\leq \varepsilon \right) \leq \exp \left\{
-C_{p,H}\varepsilon ^{-\frac{1}{2H(p-1)}}\right\} .  \label{Gbeta0}
\end{equation}
\end{lemma}

\begin{proof}
Fix an arbitrary $\beta \in S^{n-1}$. Then $\Psi _{\beta }\not\equiv 0$
since $\varphi _{i},\dots ,\varphi _{n}$ are linearly independent. Since $%
\Psi _{\beta }$ is continuous on $I$, there exists $\vec{t}_{\beta
}=(t_{\beta }^{\prime },t_{\beta })\in I$, $\delta _{\beta }$ and $\rho
_{\beta }$ such that for all $\vec{t}\in A_{\beta }:=[t_{\beta }^{\prime
}-\delta _{\beta },t_{\beta }^{\prime }+\delta _{\beta }]\times \lbrack
t_{\beta }-\delta _{\beta },t_{\beta }+\delta _{\beta }]\subset I,$ 
\begin{equation*}
\Psi _{\beta }^{2}(\vec{t})\geq \rho _{\beta }>0.
\end{equation*}%
Let $C=2\max_{i\in \left\{ 1,\dots ,n\right\} }\sup_{\vec{t}\in I}\left\vert
\varphi (\vec{t})\right\vert <\infty $. Then for any $\beta ^{\prime }\in
S^{n-1}$,  
\begin{equation*}
\left\vert \Psi _{\beta }^{2}(\vec{t})-\Psi _{\beta ^{\prime }}^{2}(\vec{t}%
)\right\vert \leq C\left\Vert \beta -\beta ^{\prime }\right\Vert .
\end{equation*}%
Then for any $\beta ^{\prime }\in S^{n-1}$ satisfying $\left\Vert \beta
^{\prime }-\beta \right\Vert \leq \rho _{\beta }/2C$, one has 
\begin{equation}
\Psi _{\beta ^{\prime }}^{2}(\vec{t})\geq \Psi _{\beta }^{2}(\vec{t}%
)-\left\vert \Psi _{\beta }^{2}(\vec{t})-\Psi _{\beta ^{\prime }}^{2}(\vec{t}%
)\right\vert \geq \rho _{\beta }/2,  \label{Psi}
\end{equation}%
for any $\vec{t}\in A_{\beta }$. 

Note that $S^{n-1}$ has a finite cover $S^{n-1}\subset \cup
_{i=1}^{m}B(\beta ^{i},\frac{\rho _{\beta ^{i}}}{2C})$. Denote $\rho
_{i}=\rho _{\beta ^{i}}$, $\delta _{i}=\delta _{\beta ^{i}}$, $\vec{t}_{i}=%
\vec{t}_{\beta ^{i}}$ and $A_{i}=A_{\beta ^{i}}$. Then it follows from (\ref%
{Psi}) that for any $\beta \in S^{n-1}$, there exists a $\beta ^{i}\in
S^{n-1}$ such that 
\begin{equation*}
\Psi _{\beta }^{2}(\vec{t})\geq \rho _{i}/2,\text{ for all }\vec{t}\in A_{i}.
\end{equation*}%
Then noticing that $\left\vert t-t^{\prime }\right\vert \leq 1$ and applying
Jensen's inequality we obtain 
\begin{eqnarray}
G_{\beta } &\geq &\frac{\rho _{i}}{2}\int_{A_{i}}\frac{\left\vert
B_{t}^{H}-B_{t^{\prime }}^{H}\right\vert ^{2p-2}}{\left\vert t-t^{\prime
}\right\vert ^{q}}d\vec{t}\geq \frac{\rho _{i}}{2}\int_{A_{i}}\left\vert
B_{t}^{H}-B_{t^{\prime }}^{H}\right\vert ^{2p-2}d\vec{t}  \notag \\
&\geq &\frac{\rho _{i}}{2(2\delta _{i})^{p-2}}\left( \int_{A_{i}}\left(
B_{t}^{H}-B_{t^{\prime }}^{H}\right) ^{2}d\vec{t}\right) ^{p-1}.
\label{Gbeta2}
\end{eqnarray}%
Note that for $f\in C[a,b]$ with average $\overline{f}=\frac{1}{b-a}%
\int_{a}^{b}f(\xi )d\xi $, we have 
\begin{equation*}
\frac{1}{b-a}\int_{a}^{b}(f(\xi )-\overline{f})^{2}d\xi \leq \frac{1}{b-a}%
\int_{a}^{b}(f(\xi )-c)^{2}d\xi 
\end{equation*}%
for any number $c$. Then 
\begin{equation}
\int_{A_{i}}\left( B_{t}^{H}-B_{t^{\prime }}^{H}\right) ^{2}d\vec{t}%
=\int_{t_{i}-\delta _{i}}^{t_{i}+\delta _{i}}\int_{t_{i}^{\prime }-\delta
_{i}}^{t_{i}^{\prime }+\delta _{i}}\left( B_{t}^{H}-B_{t^{\prime
}}^{H}\right) ^{2}dtdt^{\prime }\geq 2\delta _{i}\int_{t_{i}-\delta
_{i}}^{t_{i}+\delta _{i}}\left( B_{t}^{H}-\overline{B^{H}}\right) ^{2}dt
\label{Gbeta3}
\end{equation}%
where $\overline{B^{H}}=\int_{t_{i}-\delta _{i}}^{t_{i}+\delta
_{i}}B_{t}^{H}dt$. Combining (\ref{Gbeta2}) and (\ref{Gbeta3}) and applying
Lemma \ref{PathVar} we obtain%
\begin{eqnarray*}
P\left( G_{\beta }\leq \varepsilon \right)  &\leq &P(\int_{t_{i}-\delta
_{i}}^{t_{i}+\delta _{i}}\left( B_{t}^{H}-\overline{B^{H}}\right) ^{2}dt\leq
\left( \rho _{i}\delta _{i}\right) ^{-\frac{1}{p-1}}\varepsilon ^{\frac{1}{%
p-1}}) \\
&\leq &\exp \{-c_{H}\delta _{i}\left( \rho _{i}\delta _{i}\right) ^{\frac{1}{%
2H(p-1)}}\varepsilon ^{-\frac{1}{2H(p-1)}}\}.
\end{eqnarray*}
Then one obtains (\ref{Gbeta0}) by choosing $C_{p,H}=c_{H}\min_{1\leq i\leq
m}\delta _{i}\left( \rho _{i}\delta _{i}\right) ^{\frac{1}{2H(p-1)}}$.
\end{proof}

\textbf{Remark:} At the first glance, it seems that (\ref{Gbeta0}) can be
obtained by applying (\ref{L-S99}) to the first inequality in (\ref{Gbeta2}%
). But (\ref{L-S99}) can only be applied to square interval on the diagonal
like $[a,b]\times \lbrack a,b]$ (after applying the scaling and
self-similarity property of fBm), and here the interval $A_{i}=[t_{i}^{%
\prime }-\delta _{i},t_{i}^{\prime }+\delta _{i}]\times \lbrack t_{i}-\delta
_{i},t_{i}+\delta _{i}]$ is off diagonal.

\begin{lemma}
\label{V}For any integer $n$, the random vector $\mathbf{V}=\left(
V_{1},\dots ,V_{n}\right) $ defined in $(\ref{vi})$ is non-degenerate.
\end{lemma}

\begin{proof}
Denote by $M=\left( \left\langle DV_{i},DV_{j}\right\rangle _{\mathfrak{H}%
}\right) $ the Malliavin matrix of $\mathbf{V}$. We want to show that $%
\left( \det M\right) ^{-1}\in L^{k}$, for any $k\geq 1$. Note that $\det
M\geq \gamma _{1}^{n}$, where $\gamma _{1}>0$ is the smallest eigenvalue of
the positive definite matrix $M$. Then it suffices to show that $\gamma
_{1}^{-1}\in L^{nk}$, for any $k\geq 1$, for which it is enough to estimate $%
P(\gamma _{1}\leq \varepsilon )$ for $\varepsilon $ small. We have 
\begin{equation}
\gamma _{1}=\inf_{\left\Vert \beta \right\Vert =1}\left( M\beta ,\beta
\right) =\inf_{\left\Vert \beta \right\Vert =1}\left\Vert D\left(
\sum_{i=1}^{n}\beta _{i}V_{i}\right) \right\Vert _{\mathfrak{H}}^{2}.
\label{gamma1}
\end{equation}%
For any $\beta =(\beta _{1},\dots ,\beta _{n})\in S^{n-1}$, let $\Psi
_{\beta }(\vec{t})=\sum_{i=1}^{n}\beta _{i}\varphi _{i}(\vec{t})$. Then, 
\begin{equation*}
D_{r}\left( \sum_{i=1}^{n}\beta _{i}V_{i}\right) =\left( 2p-1\right)
\int_{I}\Psi _{\beta }(\vec{t})\frac{\left\vert B_{t}^{H}-B_{t^{\prime
}}^{H}\right\vert ^{2p-2}}{\left\vert t-t^{\prime }\right\vert ^{q}}\mathbf{1%
}_{[t^{\prime },t]}(r)d\vec{t}.
\end{equation*}%
Applying (\ref{Qexpan}) in the computation of the norm (\ref{gamma1}) yields 
\begin{eqnarray*}
\left\Vert D\left( \sum_{i=1}^{n}\beta _{i}V_{i}\right) \right\Vert _{%
\mathfrak{H}}^{2} &=&\left( 2p-1\right) ^{2}\int_{0}^{1}dr\left(
\int_{I}\Psi _{\beta }(\vec{t})\frac{\left\vert B_{t}^{H}-B_{t^{\prime
}}^{H}\right\vert ^{2p-2}}{\left\vert t-t^{\prime }\right\vert ^{q}}\mathbf{1%
}_{[t^{\prime },t]}(r)d\vec{t}\right) ^{2} \\
&=&\left( 2p-1\right) ^{2}\sum_{i\geq 1}\lambda _{i}\left( \int_{I}\varphi
_{i}(\vec{t})\Psi _{\beta }(\vec{t})\frac{\left\vert B_{t}^{H}-B_{t^{\prime
}}^{H}\right\vert ^{2p-2}}{\left\vert t-t^{\prime }\right\vert ^{q}}d\vec{t}%
\right) ^{2} \\
&\geq &\left( 2p-1\right) ^{2}\sum_{i=1}^{n}\lambda _{i}q_{i}^{2},
\end{eqnarray*}%
where $q_{i}=\int_{I}\varphi _{i}(\vec{t})\Psi _{\beta }(\vec{t})\frac{%
\left( B_{t}^{H}-B_{t^{\prime }}^{H}\right) ^{2p-2}}{\left\vert t-t^{\prime
}\right\vert ^{q}}d\vec{t}$. The definition (\ref{GbetaDef}) implies $%
G_{\beta }=\sum_{i=1}^{n}\beta _{i}q_{i}$. Since $\lambda _{1}\geq \dots
\geq \lambda _{n}>0$, we obtain 
\begin{equation*}
\sum_{i=1}^{n}\lambda _{i}q_{i}^{2}\geq \lambda
_{n}\sum_{i=1}^{n}q_{i}^{2}\geq \lambda _{n}\sum_{i=1}^{n}\beta
_{i}^{2}q_{i}^{2}\geq \frac{\lambda _{n}}{n}G_{\beta }^{2},
\end{equation*}%
where in the third inequality we used the fact that $\sum_{i=1}^{n}a_{i}^{2}%
\geq \frac{1}{n}(\sum_{i=1}^{n}a_{i})^{2}$. Therefore 
\begin{equation}
\left\Vert D\left( \sum_{i=1}^{n}\beta _{i}V_{i}\right) \right\Vert _{%
\mathfrak{H}}^{2}\geq (2p-1)^{2}\frac{\lambda _{n}}{n}G_{\beta }^{2}.
\label{gamma2}
\end{equation}%
Combining (\ref{gamma1}) and (\ref{gamma2}) we have%
\begin{equation}
\gamma _{1}=\inf_{\left\Vert \beta \right\Vert =1}\left( M\beta ,\beta
\right) \geq (2p-1)^{2}\frac{\lambda _{n}}{n}\inf_{\left\Vert \beta
\right\Vert =1}G_{\beta }^{2}.  \label{gamma}
\end{equation}

For any $\varepsilon >0$ and $0<\alpha <\frac{1}{2H(p-1)}$, let 
\begin{equation*}
W_{\beta }=\left\{ G_{\beta }\geq \varepsilon \right\} ,
\end{equation*}%
and%
\begin{equation*}
W_{n}=\left\{ \left\Vert DV_{i}\right\Vert _{\mathfrak{H}}^{2}\leq \exp
\varepsilon ^{-\alpha },i=1,\dots ,n\right\} .
\end{equation*}%
On $W_{n}$, for any $\beta ,\beta ^{\prime }\in S^{n-1}$ we have 
\begin{equation*}
\left\vert \left( M\beta ,\beta \right) -\left( M\beta ^{\prime },\beta
^{\prime }\right) \right\vert \leq C_{n}\left\Vert \beta -\beta ^{\prime
}\right\Vert \exp \frac{1}{\varepsilon ^{\alpha }},
\end{equation*}%
where $C_{n}$ is a constant independent of $\beta ,\beta ^{\prime }$ and $%
\varepsilon $.

Note that we can find a finite cover $\cup _{i=1}^{m}B(\beta ^{i},\exp (-%
\frac{2}{\varepsilon ^{\alpha }}))$ of $S^{n-1}$ with $\beta ^{i}\in S^{n-1}$
and 
\begin{equation*}
m\leq C\exp \frac{2n}{\varepsilon ^{\alpha }}.
\end{equation*}%
Then on $W_{n}$, for any $\beta \in S^{n-1}$, there exists a $\beta ^{i}$
such that 
\begin{equation*}
\left( M\beta ,\beta \right) \geq \left( M\beta ^{i},\beta ^{i}\right)
-C_{n}\exp \frac{1}{\varepsilon ^{\alpha }}\exp (-\frac{2}{\varepsilon
^{\alpha }}).
\end{equation*}%
On $W_{\beta ^{i}}\cap W_{n}$, applying (\ref{gamma}) with $A_{n}=(2p-1)^{2}%
\frac{\lambda _{n}}{n}$ and taking $\varepsilon $ small enough, 
\begin{equation*}
\left( M\beta ,\beta \right) \geq A_{n}\varepsilon ^{2}-C_{n}\exp (-\frac{1}{%
\varepsilon ^{\alpha }})\geq \frac{A_{n}}{2}\varepsilon ^{2}.
\end{equation*}%
Hence, on $\cap _{i=1}^{m}W_{\beta ^{i}}\cap W_{n}$, 
\begin{equation}
\gamma _{1}=\inf_{\left\Vert \beta \right\Vert =1}\left( M\beta ,\beta
\right) \geq \frac{A_{n}}{2}\varepsilon ^{2}>0.  \label{mu0}
\end{equation}%
On the other hand, applying Lemma \ref{Gbeta}, we have 
\begin{eqnarray}
P(\cup _{i=1}^{m}W_{\beta ^{i}}^{c}) &\leq &\sum_{i=1}^{m}P(\cup
_{i=1}^{m}W_{\beta ^{i}}^{c})\leq m\sqrt{2}\exp (-\frac{C_{p,H}}{\varepsilon
^{1/2H(p-1)}})  \notag \\
&\leq &C\exp \frac{2n}{\varepsilon ^{\alpha }}\exp (-\frac{C_{p,\alpha }}{%
\varepsilon ^{1/2H(p-1)}})\leq C\exp (-\frac{C}{\varepsilon ^{1/2H(p-1)}}).
\label{mu1}
\end{eqnarray}%
Also, by Chebyshev's inequality, we can write%
\begin{equation}
P(W_{n}^{c})\leq C\exp (-\frac{1}{\varepsilon ^{\alpha }}).  \label{mu2}
\end{equation}%
Then it follows from (\ref{mu0})--(\ref{mu2}) that for $\varepsilon $ small, 
\begin{equation*}
P(\gamma _{1}<\frac{A_{n}}{2}\varepsilon ^{2})\leq C\exp (-\frac{1}{%
\varepsilon ^{\alpha }}).
\end{equation*}%
This completes the proof of the lemma.
\end{proof}

\begin{proof}[\textbf{Proof of Theorem }$\protect\ref{Main-Non-deg}$]
Note that 
\begin{equation}
\left\Vert DF\right\Vert _{\mathfrak{H}}^{2}=16p^{2}\sum_{i\geq 1}\lambda
_{i}V_{i}^{2}\geq 16p^{2}\lambda _{n}\sum_{i=1}^{n}V_{i}^{2},  \label{DF1}
\end{equation}%
for any integer $n$. Then, denoting $\left\vert \mathbf{V}\right\vert
^{2}=\sum_{i=1}^{n}V_{i}^{2}$ we have 
\begin{equation*}
P(\left\Vert DF\right\Vert _{\mathfrak{H}}<\varepsilon )\leq P\left(
\left\vert \mathbf{V}\right\vert <\frac{\varepsilon }{4p\sqrt{\lambda _{n}}}%
\right) .
\end{equation*}%
Since $\mathbf{V}=\left( V_{1},\dots ,V_{n}\right) $ is non-degenerate, then
it has a smooth density $f_{V_{n}}(x)$. Then we have 
\begin{equation*}
P\left( \left\vert \mathbf{V}\right\vert <\frac{\varepsilon }{4p\sqrt{%
\lambda _{n}}}\right) \leq C_{n,p}\varepsilon ^{n},
\end{equation*}%
where $C_{n,p}=\frac{2\pi ^{n/2}}{n\Gamma (\frac{n}{2})}\left( 4p\sqrt{%
\lambda _{n}}\right) ^{-n}\max_{\left\vert x\right\vert \leq 1}f_{V_{n}}(x)$%
. Now the theorem follows.
\end{proof}

\textbf{Acknowledgements} F. Lu would like to thank Jared Bronski for
helpful discussions.

\end{document}